\documentclass[]{amsart}

\usepackage{amssymb,amscd} 
\usepackage{amsmath}
\usepackage{amsxtra}
\usepackage{amsfonts} 
\usepackage{amsthm}
\usepackage{epsfig,subfigure}
\usepackage{psfrag}
\usepackage{hyperref}

\theoremstyle{plain}
\newtheorem{theorem}{Theorem}[section]
\newtheorem{lemma}[theorem]{Lemma}
\newtheorem{proposition}[theorem]{Proposition}
\newtheorem{corollary}[theorem]{Corollary}

\theoremstyle{definition}
\newtheorem{remark}[theorem]{Remark}

\newcommand{\MM}{\mathcal M}
\newcommand{\BM}{\overline{\mathcal M}}
\newcommand{\bx}{\overline{x}}
\newcommand{\bM}{\overline{M}}
\newcommand{\TMM}{\widetilde{\mathcal M}}

\newcommand{\CC}{\mathcal C}

\newcommand{\OO}{\mathcal O}

\newcommand{\JJ}{\mathcal J}

\newcommand{\Eff}{\operatorname{Eff}}
\newcommand{\BEff}{\overline{\operatorname{Eff}}}

\newcommand{\age}{\operatorname{age}}
\newcommand{\Aut}{\operatorname{Aut}}
\newcommand{\sing}{\operatorname{sing}}
\newcommand{\tor}{\operatorname{tor}}

\newcommand{\irr}{\operatorname{irr}}
\newcommand{\Jac}{\operatorname{Jac}}
\newcommand{\reg}{\operatorname{reg}}

\newcommand{\bbC}{\mathbb C}

\newcommand{\bbP}{\mathbb P}

\newcommand{\bbZ}{\mathbb Z}

\title{Extremal effective divisors on $\BM_{1,n}$}

\author{Dawei Chen}

\address{Department of Mathematics, Boston College, Chestnut Hill, MA 02467}

\email{dawei.chen@bc.edu}

\author{Izzet Coskun}

\address{University of Illinois at
  Chicago, Department of Mathematics, Statistics, and Computer Science, Chicago, IL 60607}

\email{coskun@math.uic.edu}

\thanks{During the preparation of this article the first author was partially supported by NSF grant DMS-1200329 and the second author was
 partially supported by NSF CAREER grant DMS-0950951535 and an Alfred P. Sloan Foundation Fellowship.}

\begin{document}
\bibliographystyle{halpha}

\begin{abstract}
For every $n\geq 3$, we exhibit infinitely many extremal effective divisors on the moduli space of genus one curves with $n$ marked points. 
\end{abstract}

\maketitle

\setcounter{tocdepth}{1}
\tableofcontents

%%%%%%%%%%%%%%%%%%%%%%%%%%%
%%%%%%%%%%%%%%%%%%%%%%%%%%%%
\section{Introduction}
%%%%%%%%%%%%%%%%%%%%%%%%%%%
%%%%%%%%%%%%%%%%%%%%%%%%%%%

Let $\BM_{g,n}$ denote the moduli space of stable genus $g$ curves with $n$ ordered marked points. Understanding the cone of pseudo-effective divisors $\BEff(\BM_{g,n})$ is a  central problem in the birational geometry of $\BM_{g,n}$. Since the 1980s, motivated by the problem of determining the Kodaira dimension of $\BM_{g,n}$, many authors have constructed families of effective divisors on $\BM_{g,n}$. For example, Harris, Mumford and Eisenbud \cite{HarrisMumfordKodaira, HarrisKodaira, EisenbudHarrisKodaira}, using Brill-Noether and Gieseker-Petri divisors showed that $\BM_g$ is of general type for $g > 23$.  Using Kozsul divisors, Farkas \cite{FarkasKoszul} extended this result to $g \geq 22$.  Logan \cite{LoganKodaira}, using generalized 
Brill-Noether divisors, obtained similar results for the Kodaira dimension of $\BM_{g,n}$ when $n > 0$. 

Although we know many examples of effective divisors on $\BM_{g,n}$, the structure of the pseudo-effective cone $\BEff(\BM_{g,n})$ remains mysterious in general. Recently, inspired by the work of Keel and Vermeire \cite{Vermeire} on $\BM_{0,6}$, Castravet and Tevelev \cite{CastravetTevelev} constructed a sequence of non-boundary extremal effective divisors on $\BM_{0,n}$ for $n\geq 6$. For higher genera, Farkas and Verra \cite{FarkasVerraJacobian, FarkasVerraTheta} showed that certain variations of pointed Brill-Noether divisors are extremal on $\BM_{g,n}$ for $g-2\leq n \leq g$. However, for fixed $g$ and $n$, these constructions yield only finitely many extremal divisors. This raises the question whether there exist $g$ and $n$ such that  $\BEff(\BM_{g,n})$ is not finitely generated. 

Motivated by this question, in this paper we study the moduli space of genus one curves with $n$ marked points. Let ${\bf a} = (a_1, \ldots, a_n)$ be a collection of $n$ integers satisfying $\sum_{i=1}^n a_i = 0$, not all equal to zero. Define $D_{\bf a}$ in $\BM_{1,n}$ as the closure of the divisorial locus parameterizing smooth genus one curves with $n$ marked points $(E; p_1, \ldots, p_n)$ such that $\sum_{i=1}^n a_i p_i = 0$ in the Jacobian of $E$. 

Our main theorem is as follows. 

\begin{theorem}
\label{thm:main}
Suppose that $n\geq 3$ and $\gcd (a_1, \ldots, a_n) = 1$. Then $D_{\bf a}$ is an extremal and rigid effective divisor on $\BM_{1,n}$. Moreover, these 
$D_{\bf a}$'s yield infinitely many extremal rays for $\BEff(\BM_{1,n})$. Consequently, $\BEff(\BM_{1,n})$ is not finite polyhedral and $\BM_{1,n}$ is not a Mori dream space. 
\end{theorem}

The assumption $\gcd (a_1, \ldots, a_n) = 1$ is necessary to ensure that $D_{\bf a}$ is irreducible, see Section~\ref{sec:components}. Our strategy for proving the extremality of $D_{\bf a}$ is to exhibit irreducible curves $C$ Zariski dense in $D_{\bf a}$ such that $C\cdot D_{\bf a} < 0$. 

By exhibiting nef line bundles that are not semi-ample, Keel \cite[Corollary 3.1]{Keel} had already observed that $\BM_{g,n}$ cannot be a Mori dream space if $g \geq 3$ and $n \geq 1$. 

The divisor class of $D_{\bf a}$ was first calculated by Hain \cite[Theorem 12.1]{haindivisor} using normal functions. The restriction of this class to the locus of curves with rational tails was worked out by Cavalieri, Marcus and Wise~\cite{CavalieriMarcusWise} using Gromov-Witten theory. Two other proofs were recently obtained by Grushevsky and Zakharov \cite{GrushevskyZakharov} and by M\"uller \cite{Mueller}. We remark that all of them considered more general cycle classes in $\MM_{g,n}$ for $g\geq 1$, by pulling back the zero section of the universal Jacobian or the Theta divisor of the universal Picard variety of degree $g-1$.  

The symmetric group $\mathfrak{S}_n$ acts on $\BM_{1,n}$ by permuting the labeling of the marked points. Denote the quotient by $\TMM_{1,n} = \BM_{1,n} / \mathfrak{S}_n$. In contrast to Theorem \ref{thm:main}, in the last section, we show that $\BEff(\TMM_{1,n})$ is finitely generated. In fact, following an argument of Keel and M$^{\rm{c}}$Kernan \cite{KeelMcKernanContractible}, we prove that the boundary divisors generate  $\BEff(\TMM_{1,n})$. 

Note that for a subgroup $G\subset \mathfrak{S}_n$, if infinitely many irreducible divisors $D_{\bf a}$ in the above can be directly defined 
on $\BM_{1,n}/G$, then $\BEff(\BM_{1,n}/G)$ is not finitely generated. For instance, consider $n = 6$ and $G$ the subgroup of $\mathfrak{S}_6$ generated by three simple transpositions $(12)$, $(34)$ and $(56)$. Then $D_{(a,a,b,b,c,c)}$ is well-defined on $\BM_{1,6}/G$ for $a + b + c = 0$. Moreover, if $\gcd (a,b,c) = 1$, then $D_{(a,a,b,b,c,c)}$ is irreducible and extremal on $\BM_{1,6}/G$ as well. It would be interesting to classify all subgroups $G \subset \mathfrak{S}_n$ for which  
$\BEff(\BM_{1,n}/G)$ is not finitely generated. 

This paper is organized as follows. In Section~\ref{sec:prelim}, we review the divisor theory of $\BM_{1,n}$. In Section~\ref{sec:hain}, we discuss the geometry of $D_{\bf a}$, including its divisor class and irreducible components. In Section~\ref{sec:main}, we prove our main results and 
describe a conceptual understanding from the viewpoint of birational automorphisms of $\BM_{1,3}$. In Section~\ref{sec:unordered}, we study the moduli space $\TMM_{1,n}$ of genus one curves with $n$ unordered marked points and show that its effective cone is generated by boundary divisors. Finally, in the appendix, we analyze the singularities of 
$\BM_{1,n}$ and show that a canonical form defined on its smooth locus extends holomorphically to an arbitrary resolution. 

\medskip

\noindent{\bf Acknowledgements:} We would like to thank Gabi Farkas, Sam Grushevsky, Dick Hain, Joe Harris, Ian Morrison, Martin M\"{o}ller and Anand Patel for many valuable discussions about this paper. 

%%%%%%%%%%%%%%%%%%%%%%%%%%%%
%%%%%%%%%%%%%%%%%%%%%%%%%%%%
\section{Preliminaries on $\BM_{1,n}$}
\label{sec:prelim}
%%%%%%%%%%%%%%%%%%%%%%%%%%%%
%%%%%%%%%%%%%%%%%%%%%%%%%%%%

In this section, we recall basic facts concerning the geometry of $\BM_{1,n}$. We refer the reader to \cite{ArbarelloCornalba, BiniFontanariKodaira, SmythGenusOneII} for the facts quoted below. 

%% Let $\pi: \CC_{1,n} \rightarrow \BM_{1,n}$ denote the universal curve over the moduli stack and let $\omega_{\CC_{1,n}/\BM_{1,n}}$ be the dualizing sheaf. 
Let $\lambda$ be the first Chern class of the Hodge bundle on $\BM_{1,n}$. Let $\delta_{\irr}$ be the divisor class of the locus in $\BM_{1,n}$ that parameterizes curves with a non-separating node. 
The general point of  $\delta_{\irr}$ parameterizes a rational nodal curve with $n$ marked points. Let $S$ be a subset of 
$\{1,\ldots,n \}$ with cardinality $|S| \geq 2$ and let $S^c$ denote its complement. Let $\delta_{0;S}$ denote the divisor class of the locus in $\BM_{1,n}$ parameterizing curves with a node that separates the curve into a stable genus zero curve marked by $S$ and a stable genus one curve marked by $S^c$.  In addition, let $\psi_i$ be the first Chern class of the cotangent bundle associated to the $i$th marked point for $1\leq i \leq n$. Here we consider the divisor classes in the moduli stack instead of the coarse moduli scheme, see e.g.~\cite[Section 3.D]{HarrisMorrison} for more details.
 
The rational Picard group of $\BM_{1,n}$ is generated by $\lambda$ and $\delta_{0;S}$ for all $|S| \geq 2$. The divisor classes $\delta_{\irr}$ and $\psi_i$ can be expressed in terms of the generators as 
$$ \delta_{\irr} = 12 \lambda, $$
$$\psi_i = \lambda + \sum_{i\in S} \delta_{0;S}.$$ 
% Write $\psi = \sum_{i=1}^n \psi_i$, $\delta$ the total boundary and $\delta_{\red} = \delta - \delta_{\irr}$. 
The canonical class of $\BM_{1,n}$ is 
$$ K_{\BM_{1,n}} = (n- 11)\lambda + \sum_{|S| \geq 2} (|S| - 2) \delta_{0;S}.$$ 

For $n \leq 10$, $\BM_{1,n}$ is rational \cite[Theorem 1.0.1]{BelorousskiThesis}. Moreover, the Kodaira dimension of $\BM_{1,11}$ is zero and the Kodaira dimension of $\BM_{1,n}$ for $n \geq 12$ is one \cite[Theorem 3]{BiniFontanariKodaira}.\footnote{In order to study the Kodiara dimension of a singular variety, one needs to ensure that a canonical form defined in its smooth locus extends holomorphically to a resolution. Farkas informed the authors that such a verification for $\BM_{1,n}$ seems not to be easily accessible in the literature. Although the Kodaira dimension of $\BM_{1,n}$ is irrelevant for our results, we will treat this issue in the appendix by a standard argument based on the Reid-Tai criterion.} 

%Note that for $n\geq 11$, $K_{\BM_{1,n}}$ is effective. Moreover, by varying one marked point on a fixed $n$-pointed elliptic curve, one obtains a moving curve that has zero intersection with $K_{\BM_{1,n}}$, hence $K_{\BM_{1,n}}$ is not big. We conclude that $K_{\BM_{1,n}}$ lies on the boundary of $\BEff(\BM_{1,n})$ for $n\geq 11$. As a consequence $\BM_{1,n}$ is not uniruled for $n\geq 11$. 

%For $n \leq 9$, using a pencil of plane cubics one can find rational curves that pass through a general point of $\BM_{1,n}$, hence in this case $\BM_{1,n}$ is uniruled. The only uncertain case is $\BM_{1,10}$.  

%%%%%%%%%%%%%%%%%%%%%%%%%%%%%%%%%%%%%%%%%
%%%%%%%%%%%%%%%%%%%%%%%%%%%%%%%%%%%%%%%%%
\section{Geometry of $D_{\bf a}$}
\label{sec:hain}
%%%%%%%%%%%%%%%%%%%%%%%%%%%%%%%%%%%%%%%%%
%%%%%%%%%%%%%%%%%%%%%%%%%%%%%%%%%%%%%%%%%

Let ${\bf a} = (a_1, \ldots, a_n)$ be a sequence of integers, not all equal to zero, such that $\sum_{i=1}^n a_i = 0$. 
The divisor $D_{\bf a}$ in $\BM_{1,n}$ is defined as the closure of the locus 
parameterizing smooth genus one curves $E$ with $n$ distinct marked points $p_1, \ldots, p_n$ satisfying $\sum_{i=1}^n a_i p_i = 0$ in $\Jac(E)$. Equivalently, let $\JJ$ denote the universal Jacobian. 
We have a map $\MM_{1,n}\to \JJ$ induced by 
$$ (E; p_1, \ldots, p_n) \mapsto \OO_{E}\Big(\sum_{i=1}^n a_ip_i\Big). $$
Then $D_{\bf a}$ is the closure of the pullback of the zero section of $\JJ$. 

The divisor class of $D_{\bf a}$ was first calculated by Hain \cite[Theorem 12.1]{haindivisor}. 
We point out that the setting of \cite{haindivisor} is slightly different from ours. There the map $\MM_{1,n}\to \JJ$, denoted by $F_{\bf d}$, extends to $\BM_{1,n}$ as a morphism in codimension one. Hence the pullback of the zero section of $\JJ$, denoted by $F_{\bf d}^{*}\eta_1$, may contain boundary divisors. In particular, if the marked points $p_1, \ldots, p_n$ coincide on $E$, the condition $\sum_{i=1}^n a_i p_i = 0$ automatically holds by the assumption $\sum_{i=1}^n a_i = 0$. In other words,  $F_{\bf d}^{*}\eta_1$  contains the boundary divisor $\delta_{0; \{1,\ldots, n \}}$. In contrast, in our setting $D_{\bf a}$ does not contain any boundary divisors. This was already observed by Cautis \cite[Proposition 3.4.7]{CautisThesis} for the case $n=2$. In order to clarify this distinction, we will first carry out a direct calculation for the class of 
$D_{\bf a}$ and confirm that it matches with \cite{haindivisor} after adding $\delta_{0; \{1,\ldots, n \}}$. 

%%%%%%%%%%%%%%%%%%%%%%%%
\subsection{Divisor class of $D_{\bf a}$}
%%%%%%%%%%%%%%%%%%%%%%%%

Take a general one-dimensional family $\pi: \CC\to B$ of genus one curves with $n$ sections $\sigma_1, \ldots, \sigma_n$
such that every fiber 
contains at most one node and the total space of the family is smooth. Suppose there are $d_S$ fibers in which the sections labeled by $S$ intersect simultaneously and pairwise transversally. Let $d_{\irr}$ be the number of rational nodal fibers. Let $\omega$ be the first Chern class of the relative dualizing sheaf associated to $\pi$ and $\eta$ the locus of nodes in $\CC$. Then the following formulae are standard \cite{HarrisMorrison}:
$$ \pi_{*}\eta = d_{\irr} + \sum_S d_{S}, $$
$$ \omega^2 = - \sum_{S} d_{S}, $$
$$ \sigma_i \cdot \sigma_j = \sum_{\{i,j\}\subset S} d_S, $$
$$ \omega\cdot \sigma_i = - \sigma_i^2 = B\cdot \psi_i - \sum_{i\in S} d_{S} = \frac{1}{12} d_{\irr}. $$ 

Suppose 
$D_{\bf a}$ has class 
$$ D_{\bf a} = c_{\irr} \delta_{\irr} + \sum_{|S|\geq 2} c_{S} \delta_{0; \{ S\}} $$
with unknown coefficients $c_{\irr}$ and $c_{S}$. By \cite[page 11]{GrushevskyZakharov}, the zero section of $\JJ$ 
vanishes along the boundary divisor $\delta_{0; \{1, \ldots, n \}}$ with multiplicity one.  
Applying the Grothendieck-Riemann-Roch formula to the push-forward of the section $\sum_{i=1}^n a_i \sigma_i$, we conclude that 
\begin{eqnarray*} 
B \cdot D_{\bf a} + d_{\{1, \ldots, n \}} & = & c_1\Big(R^1\pi_{*}\sum_{i=1}^n a_i \sigma_i\Big) \\
        & = & - \pi_{*}\Big(\Big(  1 + \sum_{i=1}^n a_i \sigma_i + \frac{1}{2}\Big( \sum_{i=1}^n a_i \sigma_i\Big)^2\Big)\Big( 1 - \frac{\omega}{2} + \frac{\omega^2 + \eta}{12}\Big)\Big) \\
        & = & -\frac{1}{12} d_{\irr} + \frac{1}{24}\Big(\sum_{i=1}^n a_i^2\Big) d_{\irr} - \sum_S\sum_{\{i,j \}\subset S} a_i a_j d_S. 
\end{eqnarray*}
%The extra term $d_{0; \{1, \ldots, n \}}$ on the left hand side is because when the $n$ marked points coincide, the corresponding locus is automatically contained in the tautological locus. When we blow it up, this gives rise to the boundary component $\delta_{0; \{1, \ldots, n \}}$. 

By comparing coefficients on the two sides of the equation, we obtain that 
$$ 12 c_{\irr} = -1 + \frac{1}{2}\sum_{i=1}^n a_i^2, $$
$$ c_{S} =  - \sum_{\{i,j \}\subset S} a_i a_j, \ S\neq \{1,\ldots, n\}, $$
$$ c_{\{1,\ldots,n\}} = - \sum_{1\leq i < j \leq n} a_i a_j -1 = -1 +  \frac{1}{2}\sum_{i=1}^n a_i^2, $$
where the last equality uses the assumption $\sum_{i=1}^n a_i = 0$. 
Hence, we conclude the following. 
\begin{proposition}
\label{prop:totalclass}
The divisor class of $D_{\bf a}$ is given by  
$$ D_{\bf a} = \Big(-1 + \frac{1}{2}\sum_{i=1}^n a_i^2\Big) (\lambda + \delta_{0; \{1,\ldots, n\}}) - \sum_{2 \leq |S| < n} \Big(\sum_{\{i,j \}\subset S} 
a_ia_j\Big) \delta_{0;S}. $$
\end{proposition}

Therefore, adding $\delta_{0; \{1, \ldots, n \}}$ to $D_{\bf a}$, we recover the divisor class calculated in \cite[Theorem 12.1]{haindivisor}. 

%%%%%%%%%%%%%%%%%%%%%%%%%%%%%%
\subsection{Irreducible components of $D_{\bf a}$}
\label{sec:components}
%%%%%%%%%%%%%%%%%%%%%%%%%%%%%%

The divisor $D_{\bf a}$ is not always irreducible. For instance for $D_{(4,-4)}$ on $\BM_{1,2}$, the condition is $4p_1 - 4p_2 = 0$. There are two possibilities, $2p_1 - 2p_2 = 0$ and $2p_1 - 2p_2 \neq 0$, each yielding a component for $D_{(4,-4)}$. 
In general for $n\geq 3$, if $\gcd (a_1, \ldots, a_n) = 1$, then $D_{\bf a}$ is irreducible. If $\gcd (a_1, \ldots, a_n)  > 1$, then $D_{\bf a}$ contains more than one component. Below we will prove this statement and calculate the divisor class of each irreducible component. 

First, consider the special case $n=2$. Let $\eta(d)$ denote the number of positive integers that divide $d$. 

\begin{proposition}
\label{prop:component-2}
Suppose $a$ is an integer bigger than one. Then the divisor $D_{(a,-a)}$ in $\BM_{1,2}$ consists of $\eta(a) - 1$ irreducible components. 
\end{proposition}

\begin{proof}
By definition, $D_{(a,-a)}$ is the closure of the locus parameterizing $(E; p_1, p_2)$ such that $p_2 - p_1$ is a non-zero $a$-torsion. Take the square 
$[0, a]\times [0,ai]$ and glue its parallel edges to form a torus $E$. Fix $p_1$ as the origin of $E$. The number of $a$-torsion points $p_2$ is equal to $a^2$ and the coordinates $(x, y)$ of each $a$-torsion point satisfy $x, y\in \bbZ / a$. 

When varying the lattice structure of $E$, the monodromy group acts on $(x,y)$. Suppose we fix the horizontal edge and shift the vertical edge to the right until we obtain a parallelogram spanned by $[0,a]\times [0, a(1+i)]$. The resulting torus is isomorphic to $E$. Consequently the monodromy action sends an $a$-torsion point $(x, y)$ to $(x+y, y)$. Similarly, we may also obtain the action sending $(x, y)$ to $(x, x+y)$. Then each orbit of the monodromy action is uniquely determined by $k = \gcd (x, y, a)$. In other words, the monodromy is transitive on the primitive $a'$-torsion points for each divisor $a'$ of $a$, where $a' = a / k$. 
Hence, the number of its orbits is $\eta(a)$. Each orbit gives rise to an irreducible component of $D_{(a,-a)}$ parameterizing $(E, p_1, p_2)$ such that $p_2 - p_1$ is a primitive $a'$-torsion, where $a$ is divisible by $a'$. Moreover, when $a'=1$, i.e. $p_2 = p_1$, the corresponding component is $\delta_{0; \{1,2 \}}$, hence we need to exclude it by our setting. 
\end{proof}

Next, we consider the case $n\geq 3$. If $m$ entries of ${\bf a}$ are zero, drop them and denote by ${\bf a}'$ the resulting $(n-m)$-tuple. Then we have $D_{\bf a} = \pi^{*} D_{{\bf a}'}$, where $\pi: \BM_{1,n}\to \BM_{1,n-m}$ is the map forgetting the corresponding marked points. Since the fiber of $\pi$ over a general point in $D_{{\bf a}'}$ is irreducible, we conclude that $D_{\bf a}$ and $D_{{\bf a}'}$ possess the same number of irreducible components. It remains to consider the case when all entries of ${\bf a}$ are non-zero. 

\begin{proposition}
\label{prop:component-n}
Suppose $n\geq 3$ and all entries of ${\bf a}$ are non-zero. Let $d = \gcd (a_1, \ldots, a_n)$. Then $D_{\bf a}$ consists of $\eta(d)$ irreducible components. 
\end{proposition}

\begin{proof}
If an entry of ${\bf a}$ equals $1$ or $-1$, say $a_n = 1$, then we can freely choose $p_1, \ldots, p_{n-1}$ and a general choice uniquely determines $p_n$. In other words, $D_{\bf a}$ is birational to $\MM_{1,n-1}$ which is irreducible. 

Suppose all the entries are different from $1$ and $-1$. Fix $p_1, \ldots, p_{n-1}$ and replace $p_{n}$ by $p'_{n} = 2p_{1} - p_n$, then ${\bf a} = (a_1, \ldots, a_n)$ becomes ${\bf a}' = (a_1 + 2a_n, a_2, \ldots, a_{n-1}, -a_n)$. Note that $p_n$ and $p'_n$ uniquely determine each other, and for general points in $D_{\bf a}$ we have $p'_n \neq p_i$ for $1\leq i < n$, otherwise we would have $|a_i| = |a_n| = 1$. Hence the components of $D_{\bf a}$ and $D_{{{\bf a}'}}$ have a one to one correspondence. Using such transformations, we can decrease $\min\{|a_1|, \ldots, |a_n|\}$ until one of the entries is equal to $d$, say $a_n =  d$. 

Now fix $p_1, \ldots, p_{n-1}$ and set $\sum_{i=1}^{n-1} a_i p_i$ to be the origin of $E$. Then $p_n$ is a $d$-torsion. Analyzing the monodromy associated 
to $D_{\bf a}\dashrightarrow \MM_{1,n-1}$ as in the proof of Proposition~\ref{prop:component-2}, we see that $D_{\bf a}$ has at most $\eta(d)$ irreducible components. On the other hand for each positive factor $s$ of $d$, the locus parameterizing $\sum_{i=1}^n b_i p_i = 0$ where $b_i = a_i / s$ gives rise to at least one component of $D_{\bf a}$. Hence $D_{\bf a}$ contains exactly $\eta(d)$ irreducible components. Since $n\geq 3$, none of these components is a boundary divisor of $\BM_{1,n}$. 
\end{proof}

Let us calculate the divisor class of each component of $D_{\bf a}$. As in the proof of Proposition~\ref{prop:component-n}, let 
$D'_{{\bf a}}$ be the irreducible component of $D_{\bf a}$ such that $\sum_{i=1}^n a_i p_i = 0$ but 
$\sum_{i=1}^n (a_i/s) p_i \neq 0$ for general points in $D'_{{\bf a}}$ and any $s$ dividing $\gcd (a_1, \ldots, a_n)$. 

\begin{proposition}
\label{prop:componentclass}
Suppose $\gcd (a_1, \ldots, a_n) = d > 1$. Then the divisor $D'_{{\bf a}}$ has class 
$$ D'_{{\bf a}} = \prod_{p|d} \Big(1 - \frac{1}{p^2}\Big) (D_{\bf a} +   \lambda + \delta_{0; \{1,\ldots, n\}}),$$
where the product ranges over all primes $p$ dividing $d$. 
\end{proposition}

We remark that for $n=2$ the above divisor class was calculated by Cautis \cite[Proposition 3.4.7]{CautisThesis} and also communicated personally to the authors by Hain. 

\begin{proof}
Let $b_i = a_i / d$ and ${\bf b} = (b_1, \ldots, b_n)$. For an integer $m$, use the notation $m {\bf b} = (mb_1, \ldots, mb_n)$. 
Note that 
$$ D_{\bf a} = D_{d{\bf b}} = \sum_{t | d} D'_{t{\bf b}}, $$
where $t$ ranges over all positive integers dividing $d$.  
By Proposition~\ref{prop:totalclass}, we have 
$$ D_{\bf a} +  \lambda + \delta_{0; \{1,\ldots, n\}} = d^2 (D_{\bf b} +  \lambda + \delta_{0; \{1,\ldots, n\}}). $$
%Note that 
%$$D'_{\underline{b}} = D_{\underline{b}} = F_{\underline{b}} - \lambda - \delta_{0; \{1,\ldots, n\}}. $$ 
For an integer $t\geq 2$, define 
$$\sigma(t) =  t^2 \prod_{p|t} \Big(1 - \frac{1}{p^2}\Big),$$
where the product ranges over all primes $p$ dividing $t$. We also set $\sigma(1) = 1$. 
Using the above observation, it suffices to prove that  
$$ \sum_{t | d }\sigma(t) = d^2 $$
for all $d$.  

In order to prove the above equality, do induction on $d$. Write $d$ as 
$$d = q^m e, $$ 
where $q$ is a prime and $e$ is not divisible by $q$. Let $S_i$ be the set of positive integers $t$ dividing $d$, such that $t$ is divisible by $q^i$ but not by $q^{i+1}$ for any $1\leq i \leq m$. By induction, we have 
$$ \sum_{t\in S_i} \sigma(t) = q^{2i}\Big(1-\frac{1}{q^2}\Big) e^2. $$
Summing over all $i$, we thus obtain that 
$$ \sum_{t|d} \sigma(t) = q^{2m} e^2 = d^2. $$
\end{proof}

\begin{corollary}
If $\gcd (a_1, \ldots, a_n) > 1$, the divisor class $D'_{{\bf a}}$ is not extremal. 
\end{corollary} 

\begin{proof}
By Proposition~\ref{prop:componentclass}, $D'_{{\bf a}}$ is a positive linear combination of effective divisor classes, not all proportional. 
\end{proof}

%%%%%%%%%%%%%%%%%%%%%%%%%%%%
%%%%%%%%%%%%%%%%%%%%%%%%%%%%
\section{Extremality of $D_{\bf a}$} 
\label{sec:main}
%%%%%%%%%%%%%%%%%%%%%%%%%%%%
%%%%%%%%%%%%%%%%%%%%%%%%%%%%

In this section, we will prove Theorem~\ref{thm:main}. Recall that an effective divisor $D$ in a projective variety $X$ is called extremal, if for any linear combination $D = a_1D_1 + a_2D_2$ with $a_i > 0$ and $D_i$ pseudo-effective, $D$ and $D_i$ are proportional. In this case, we say that $D$ spans an extremal ray of the pseudo-effective cone $\BEff(X)$. Furthermore, we say that $D$ is rigid, if for every positive integer $m$ the linear system $|mD|$ consists of the single element 
$mD$. An irreducible effective curve contained in $D$ is called a moving curve in $D$, if its deformations cover a dense subset of $D$. 

Let us first give a method to test the extremality and rigidity for an effective divisor. 

\begin{lemma}
\label{lem:negative}
Suppose that $C$ is a moving curve in an irreducible effective divisor $D$ satisfying $C\cdot D < 0$. Then $D$ is extremal and rigid. 
\end{lemma}

\begin{proof}
Let us first prove the extremality of $D$. Suppose that $D = a_1D_1 + a_2D_2$ with $a_i >0$ and $D_i$ pseudo-effective. If $D_1$ and $D_2$ are not proportional to $D$, we can assume that they lie in the boundary of $\BEff(X$) and moreover that $D_i - \epsilon D$ is not pseudo-effective for any $\epsilon > 0$. Otherwise, we can replace $D_1$ and $D_2$ by the intersections of their linear span with the boundary of $\BEff(X)$. 

Since $C\cdot D<0$, at least for one of the $D_i$'s, say $D_1$, we have $C\cdot D_1 < 0$. Without loss of generality, rescale the class of $D_1$ such that 
$C\cdot D_1 = - 1$. Take a very ample divisor class $A$ and consider the class $F_n = n D_1 + A$ for $n$ sufficiently large. Then $F_n$ can be represented by an effective divisor. Suppose $C\cdot A = a$ and $C\cdot D = - b$ for some $a, b > 0$. Note that if $C$ has negative intersection with an effective divisor, then it is contained in that divisor. Since $C$ is moving in $D$, it further implies that $D$ is contained in that divisor. It is easy to check that 
$C\cdot (F_n - k D) < 0$ for any $k < (n-a)/b$. Moreover, the multiplicity of $D$ in the base locus of $F_n$ is at least equal to $(n-a)/b$. 
Consequently $E_n = F_n - (n-a)D/b$ is a pseudo-effective divisor class. As $n$ goes to infinity, 
the limit of $ E_n /n$ has class $D_1 - D/b$. Since $E_n$ is pseudo-effective, we conclude that $D_1 - D/b$ is also pseudo-effective, contradicting the assumption 
that $D_1 - \epsilon D$ is not pseudo-effective for any $\epsilon > 0$. 

Next, we prove the rigidity. Suppose for some integer $m$ there exists another effective divisor $D'$ such that $D' \sim mD$. Without loss of generality, assume that $D'$ does not contain $D$, for otherwise we just subtract $D$ from both sides. Since $C\cdot D < 0$, we have $C\cdot D' < 0$, and hence $D'$ contains $C$. But $C$ is moving in $D$, hence 
$D'$ has to contain $D$, contradicting the assumption.  
\end{proof}

Although we can give a uniform proof of Theorem~\ref{thm:main} as in Section~\ref{sec:n>3}, for the reader to get a feel, let us first discuss the case $n=3$ in detail. 

%%%%%%%%%%%%%%%%%%%%%%%
\subsection{Geometry of $\BM_{1,3}$} 
\label{sec:n=3}
%%%%%%%%%%%%%%%%%%%%%%%

Let ${\bf a} = (a_1, a_2, a_3)$. If $a_3 = 0$, then $a_2 = -a_1$ are not relatively prime unless they are $1$ and $-1$. But then $p_1 = p_2$, hence the locus corresponds to the boundary divisor $\delta_{0; \{1,2\}}$. Therefore, below we assume that $\gcd(a_1, a_2, a_3) = 1$ and none of the $a_i$'s is zero. 

%Let $X$ be the closure of the locus in $\BM_{1,3}$
%parameterizing a fixed smooth genus one curve $E$ with three marked points $p_1, p_2, p_3$ such that $\sum_{i=1}^3 a_i p_i = 0$. Note that $X$ is an irreducible curve whose deformations cover an open dense subset of $D_{\bf a}$ by varying the complex structure of $E$. 

Fix a smooth genus one curve $E$ with a marked point $p_1$. Vary two points $p_2, p_3$ on $E$ such that $\sum_{i=1}^3 a_i p_i = 0$ in the Jacobian of $E$. Let $X$ be the curve induced in $\BM_{1,n}$ by this one parameter family of three pointed genus one curves. We obtain deformations of $X$ by varying the complex structure on $E$. Since these deformations cover a Zariski dense subset of $D_{\bf a}$,  we obtain a moving curve in the divisor $D_{\bf a}$.
We have the following intersection numbers:  
$$ X\cdot \delta_{\irr} = 0, $$
$$X\cdot \delta_{0;\{i,j\}} = a_k^2 - 1 \ \mbox{for}\ k\neq i, j, $$
$$ X\cdot \delta_{0;\{1,2,3\}} = 1. $$

The intersection numbers $X \cdot \delta_{\irr}$ and $X \cdot \delta_{0; \{i,j \}}$ are straightforward. At the intersection with $\delta_{0;\{1,2,3\}}$, $p_1, p_2, p_3$ coincide at the same point $t$ in $E$. Blow up $t$ and we obtain a rational tail $R\cong \bbP^1$ that contains the three marked points. Without loss of generality, suppose 
$a_1 > 0$ and $a_2, a_3 < 0$. The pencil induced by $a_1 p_1 \sim (-a_2)p_2 + (-a_3) p_3$ degenerates to an admissible cover $\pi$ of degree $a_1$. By the  Riemann-Hurwitz formula, $\pi$ is totally ramified at $p_1$, has ramification order $(-a_i)$ at $p_i$ for $i = 2, 3$, and is simply ramified at $t$. Suppose $\pi(p_1) = 0$, $\pi(p_2) = \pi(p_3) = \infty$ and $\pi(t) = 1$ in the target $\bbP^1$. Then in affine coordinates $\pi$ is given by 
$$\pi(x) = \prod_{i=1}^3(x-p_i)^{a_i}. $$
The condition imposed on $t$ is that 
$$(x-p_1)^{a_1} - (x-p_2)^{-a_2}(x-p_3)^{-a_3}$$ 
has a critical point at $t$ and $\pi(t) =1$. Solving for $t$, we easily see that $t$ 
exists and is uniquely determined by $p_1, p_2, p_3$, namely, the four points $t, p_1, p_2, p_3$ have unique moduli in $R$. 
 
Now we can prove Theorem~\ref{thm:main} for the case $n=3$. 

\begin{proof}
Using the divisor class $D_{\bf a}$ in Proposition~\ref{prop:totalclass} and the above intersection numbers, we see that 
$$X\cdot D_{\bf a}  = - 1.$$ 
By assumption both $X$ and 
$D_{\bf a}$ are irreducible. Moreover, $X$ is a moving curve inside $D_{\bf a}$. Therefore, by Lemma~\ref{lem:negative} $D_{\bf a}$ is an extremal and rigid divisor. 

To see that we obtain infinitely many extremal rays of $\BEff(\BM_{1,3})$ this way, let us take ${\bf a} = (n+1, -n, -1)$. Then $D_{(n+1, -n, -1)}$ is irreducible and its divisor class lies on the ray $$c \left( \lambda + \delta_{0,\{1,2,3\}} + \delta_{0,\{1,2\}} + \frac{1}{n} \delta_{0,\{1,3\}} - \frac{1}{n+1} \delta_{0,\{2,3\}} \right), \ \ c>0.$$ As $n$ varies, we obtain infinitely many extremal rays. 
\end{proof}

Next we give a conceptual explanation of the extremality in terms of birational automorphisms of $\BM_{1,3}$. The idea is as follows. We want to find a birational map 
$$f: \BM_{1,3} \dashrightarrow \BM_{1,3}$$ 
such that $f$ and its inverse do not contract any 
divisor, then $f$ preserves the structure of $\Eff(\BM_{1,3})$, i.e. a divisor $D$ is extremal if and only if $f_{*}(D)$ is extremal. Moreover, if $f$ sends $D$ 
to a boundary divisor $\delta_{0;S}$, then $D$ is extremal, since we know $\delta_{0; S}$ is extremal.  

A prototype of such birational automorphisms can be defined as 
$$f: (E; p_1, p_2, p_3) \mapsto (E; q_1, q_2, q_3)$$
such that 
$$q_1 = p_1, $$
$$q_2 = p_2, $$
$$ q_3 = p_2 + p_3 - p_1, $$
where $E$ is a smooth genus one curve with three marked points in general position. Then $f^{-1}$ is accordingly given by 
$$p_1 = q_1, $$ 
$$p_2 = q_2, $$
$$ p_3 = q_1 + q_3 - q_2. $$

Note that $f$ does not extend to a regular map on $\BM_{1,3}$, see Remark~\ref{rem:undefined}. But one can extend $f$ to a regular map in codimension-one, which we still denote by $f$. 

\begin{proposition}
\label{prop:auto}
Away from $D_{(2,-1,-1)}$ and the boundary of $\BM_{1,3}$, $f$ is injective with image contained in $\MM_{1,3}$. 
For general points in each boundary component of $\BM_{1,3}$, $f$ induces the following action: 
$$\delta_{\irr} \mapsto \delta_{\irr}, $$
$$  \delta_{0; \{1,2 \}} \mapsto  \delta_{0; \{1,2\}}, $$ 
$$  \delta_{0; \{1,3 \}} \mapsto  \delta_{0; \{2,3\}}, $$ 
$$  \delta_{0; \{2,3 \}} \mapsto  D_{(-1, 2, -1)}, $$ 
$$   \delta_{0; \{1,2,3 \}} \mapsto  \delta_{0; \{1,2, 3\}}. $$
For general points in $D_{(2,-1,-1)}$, the action induced by $f$ is:    
$$ D_{(2,-1,-1)} \mapsto \delta_{0; \{1,3 \}}. $$ 
\end{proposition}

\begin{proof}
Take a smooth genus one curve $E$ with three distinct marked points $p_1, p_2, p_3$. By definition, we know $q_1\neq q_2$. If $q_2 = q_3$, we get $p_3 = p_1$, contradicting the assumption. If $q_1 = q_3$, we get $2p_1 = p_2 + p_3$, i.e. $(E; p_1, p_2, p_3)$ is contained in $D_{(2,-1,-1)}$. In the complement 
$\MM_{1,3}\backslash D_{(2,-1,-1)}$, it is clear that $f$ is an injection. 

Now let us study the extension of $f$ at the boundary. Note that $p_3$ is sent to its conjugate $q_3$ under the double cover $E\to \bbP^1$ induced by the pencil $|p_1 + p_2|$. Using admissible covers, the conjugate $q_3$ is uniquely determined on a rational one-nodal curve when $p_1, p_2, p_3$ are fixed, distinct and away from the node. Therefore, we conclude that $f$ can be extended to a birational map from $\delta_{\irr}$ to itself. 

Next, consider $\delta_{0; \{1,2\}}$. Take a general point $x = (E\cup_t R; p_1, p_2, p_3)$ in $\delta_{0; \{1,2\}}$, where $t$ is the node, $E$ contains $p_3$ and the rational tail $R$ contains $p_1, p_2$. Blow down $R$ and $p_1, p_2$ stabilize to $t$. By definition, $q_3 = t + p_3 - t = p_3$ is contained in $E$. The rational tail $R$ is still stable containing $q_1 =p_1$  and $q_2 = p_2$. Hence we conclude that $f(x) = (E\cup_t R; q_1, q_2, q_3)\in \delta_{0; \{1,2\}}$, where $E$ contains $q_3$ only. The same argument can be applied to $\delta_{0; \{1,3\}}$ and we leave it to the reader.  

Take a general point $y = (E\cup_t R; p_1, p_2, p_3)$ in $\delta_{0; \{2,3\}}$, where $t$ is the node, $E$ contains $p_1$ and the rational tail $R$ contains $p_2, p_3$. Blow down $R$ and $p_2, p_3$ stabilize to $t$. By definition, $q_3 = t + t - p_1 = 2q_2 - q_1$, i.e. $q_1 - 2q_2 + q_3 = 0$. Therefore, we conclude that $f(y)$ is contained in $D_{(1,-2,1)}$, where $q_2 = t$ in $E$. 

For $\delta_{0; \{1,2,3\}}$, take a one-dimensional family of genus one curves with sections $P_1=Q_1$, $P_2=Q_2$, $P_3$ and $Q_3$ such that in a generic fiber $p_2 + p_3 = p_1 + q_3$ and all the sections meet at the central fiber. Suppose $t$ is the base parameter and $z$ is the vertical parameter. 
Let $c = (0, 0)$ be the common point of the sections in the central fiber $E$ defined by $t = 0$. Without loss of generality, around $c$ we can parameterize the tangent directions of $P_i$ by $z = 0$, $z = t z_2$ and $z= tz_3$ for $i=1,2,3$, respectively, and $z = t(z_2 + z_3)$ for $Q_3$, where 
$z_2, z_3$ are fixed in $E \cong \bbC/\bbZ^2$. Blow up $c$ and for the resulting surface, use $(t, z, [u, v])$ as the new coordinates such that $t u = z v$. Then the exceptional curve $R$ is defined by $t = z = 0$ and the proper transform 
of $E$, still denoted by $E$, is parameterized by $(0, z, [1, 0])$. In particular, $R$ and $E$ meet at $r = (0,0,[1,0])$. The proper transforms 
of the four sections meet $R$ at $p_1 = [0,1]$, $p_2 = [z_2, 1]$, $p_3 = [z_3, 1]$ and $q_3 = [z_2 + z_3, 1]$. Let $x = u/v$ be the affine coordinate 
of $R \backslash s$, where $s$ corresponds to $x = \infty$. Then there exists a unique double cover $\pi: R\to \bbP^1$ by $x\mapsto (x-z_2)(x-z_3)$ (modulo isomorphisms of $\bbP^1$) such that $\pi(p_2) = \pi(p_3)$, $\pi(p_1) = \pi(q_3)$ and $\pi$ is ramified at $r$. In other words, 
using the pencil $|2q|$ on $E$ and $\pi$ on $R$, one can construct an admissible double cover $E\cup_r R\to \bbP^1 \cup \bbP^1$ such that up to isomorphism $q_3$ in the rational tail $R$ is uniquely determined by $p_1, p_2$ and $p_3$. 

Finally, take a general point $(E; p_1, p_2, p_3)$ in $D_{(2,-1,-1)}$, i.e. $2p_1 - p_2 - p_3 = 0$. Then we conclude that 
$$q_3 = p_2 + p_3 - p_1 = p_1 = q_1.$$ 
Blow up the point where $q_1$ and $q_3$ meet. We end up with a general point in $\delta_{0; \{1,3\}}$, since three special points in $\bbP^1$ have unique moduli. 
\end{proof}

By the same token, one can prove the following for $f^{-1}$. 

\begin{proposition}
\label{prop:autoinverse}
Away from $D_{(-1, 2,-1)}$ and the boundary of $\BM_{1,3}$, $f^{-1}$ is injective with image contained in $\MM_{1,3}$. 
For general points in each boundary component of $\BM_{1,3}$, $f^{-1}$ induces the following action: 
$$\delta_{\irr} \mapsto \delta_{\irr}, $$
$$  \delta_{0; \{1,2 \}} \mapsto  \delta_{0; \{1,2\}}, $$ 
$$  \delta_{0; \{1,3 \}} \mapsto  D_{( 2, -1, -1)}, $$
$$  \delta_{0; \{2,3 \}} \mapsto  \delta_{0; \{1,3\}}, $$ 
$$   \delta_{0; \{1,2,3 \}} \mapsto  \delta_{0; \{1,2, 3\}}. $$
For general points in $D_{(-1, 2,-1)}$, the action induced by $f$ is:    
$$ D_{(-1,2,-1)} \mapsto \delta_{0; \{2,3 \}}. $$ 
\end{proposition}

\begin{corollary}
\label{cor:autoeff}
The maps $f$ and $f^{-1}$ induce isomorphisms in codimension-one. In particular, they preserve the structure of $\BEff(\BM_{1,3})$. As a consequence 
$D_{(2,-1,-1)}$ is an extremal effective divisor. 
\end{corollary}

\begin{proof}
The statement about $f$ and $f^{-1}$ is obvious by Propositions~\ref{prop:auto} and \ref{prop:autoinverse}. Since $f_{*}D_{(2,-1,-1)} = \delta_{0; \{ 1,3\}}$ is extremal and rigid, we thus conclude the extremality and rigidity for $D_{(2,-1,-1)}$. 
\end{proof}

\begin{remark}
\label{rem:undefined}
The map $f$ is not regular at the locus parameterizing two rational curves $X$ and $Y$ intersecting at two nodes $s$ and $t$, 
where $p_2, p_3$ are contained in $X$ and $p_1$ is contained in $Y$. Using admissible covers, the point $q_3$ in $Y$ satisfies $p_1 + q_3 \sim s +t$, but any point in $Y$ (away from $s$ and $t$) can be such $q_3$ because $Y$ is rational. The resulting covering curve still keep $q_1 = p_1$ and $q_3$ in $Y$, but along with $s, t$ the four special points in $Y$ have varying moduli. Therefore, its image under $f$ cannot be uniquely determined. 
\end{remark}

Using $f$ and the action of $\mathfrak{S}_3$ permuting the marked points, the signature $(a_1, a_2, a_3)$ can be sent to 
$$(a'_1, a'_2, a'_3) = (a_1 - a_3, a_2 + a_3, a_3). $$ 
Given $a_1 + a_2 + a_3 = 0$ and 
$\gcd(a_1, a_2, a_3) = 1$, without loss of generality we can assume that $a_1 > a_3 > 0$ (unless $a_1 = a_3 = 1$)
and $a_2 < 0$. Then $-a_2 = a_1 + a_3$ and $|a_3| < |a_1| < |a_2|$. The new signature satisfies 
$ |a'_i| < |a_2| $ for $1\leq i\leq 3$. Keep using such actions and eventually we can reduce the signature to ${\bf a} = (1, 1, -2)$. By Corollary~\ref{cor:autoeff}
we thus obtain another proof for Theorem~\ref{thm:main} in the case of $n=3$. 

%%%%%%%%%%%%%%%%%%%%%%%%
\subsection{Geometry of $\BM_{1,n}$ for $n\geq 4$}
\label{sec:n>3}
%%%%%%%%%%%%%%%%%%%%%%%%

In this section suppose $n\geq 4$. First, let us consider pulling back divisors from $\BM_{1,3}$. 

Let $\pi: \BM_{1,n}\to \BM_{1,3}$ be the forgetful map forgetting $p_4, \ldots, p_n$. Assume that $\gcd(a_1, a_2) = 1$.  
In Section~\ref{sec:n=3} we have shown that $D_{(a_1, a_2, -a_1-a_2)}$ is extremal. Now fix a smooth genus one curve $E$ with fixed $p_3, p_4, \ldots, p_n$ in general position. Varying $p_1, p_2$ in $E$ such that $\sum_{i=1}^3 a_i p_i = 0$, we obtain a curve $X$ moving inside $\pi^{*}D_{(a_1, a_2, -a_1 - a_2)}$. 
We have also seen that $(\pi_{*}X)\cdot D_{(a_1, a_2, -a_1 - a_2)} < 0$ on $\BM_{1,3}$, hence by the projection formula, we have 
$X\cdot (\pi^{*}D_{(a_1, a_2, -a_1 - a_2)}) < 0$. Since $\pi^{*}D_{(a_1, a_2, -a_1 - a_2)}$ is irreducible, we conclude the following. 

\begin{proposition}
Let ${\bf a} = (a_1, a_2, -a_1 - a_2, 0, \ldots, 0)$ for $\gcd (a_1, a_2) = 1$. Then the divisor class $D_{\bf a}$ 
 is extremal in $\BEff(\BM_{1,n})$. 
\end{proposition}

\begin{corollary}
For $n\geq 4$, the cone $\BEff(\BM_{1,n})$ is not finite polyhedral. 
\end{corollary}

\begin{proof}
We have 
$$ \pi^{*}\lambda = \lambda, $$
$$ \pi^{*} \delta_{0; \{1,2\}} = \sum_{\{1,2\} \subset S \atop 3\not\in S} \delta_{0;S}, $$
$$ \pi^{*} \delta_{0; \{1,2, 3\}} = \sum_{\{1,2,3\}\subset S} \delta_{0; S}. $$
Then for $\gcd (a_1, a_2) = 1$, we obtain that 
$$ \pi^{*} D_{(a_1, a_2, -a_1 - a_2)} = (-1 + a_1^2 + a_2^2 + a_1 a_2) \Big(\lambda + \sum_{\{1,2,3\}\subset S} \delta_{0; S}\Big) $$
$$ - a_1 a_2 \Big(\sum_{\{1,2\} \subset S \atop 3\not\in S}\delta_{0;S}\Big) + a_1 (a_1 + a_2) \Big(\sum_{\{1,3\} \subset S \atop 2\not\in S}\delta_{0;S}\Big) + a_2 (a_1+a_2) \Big(\sum_{\{2,3\} \subset S \atop 1\not\in S}\delta_{0;S}\Big). $$
 By varying $a_1, a_2$, we obtain infinitely many extremal rays.  
\end{proof}

Next we consider $D_{\bf a}$ when at least four entries are non-zero and $\gcd(a_1,\ldots, a_n)=1$.  
Let $D_{\bf a}(E, \eta)$ be the closure of the locus parameterizing 
$(E; p_1, \ldots, p_n)$ such that $\sum_{i=1}^n a_i p_i = \eta$ for fixed $\eta\in \Jac(E)$ on a fixed genus one curve $E$. 

For $S = \{i_1, \ldots, i_k\}$, consider the locus $\delta_{0;S}(E)$ of 
curves parameterized in $\delta_{0;S}$ whose genus one component is $E$. Blow down the rational tails and $p_{i_1}, \ldots, p_{i_k}$ reduce to the same point $q$ in $E$. 
For $\eta \neq 0$, the condition 
$$\Big(\sum_{j=1}^k a_{i_j}\Big)q + \sum_{j\not\in S} a_j p_j = \eta$$ 
does not hold for $q$ and $p_j$ in general position in $E$. Therefore, $\delta_{0;S}(E)$ is not contained in $D_{\bf a}(E, \eta)$ 
for $\eta\neq 0$, and $D_{\bf a}(E, \eta)$ is irreducible of codimension-two in $\BM_{1,n}$. 

If $\eta = 0$, the above argument still goes through with only one exception when $S = \{ 1, \ldots, n \}$, because the condition $\sum_{i=1}^n a_i p_i = 0$ automatically holds if all the marked points coincide due to the assumption $\sum_{i=1}^n a_i = 0$. In other words, $D_{\bf a}(E, 0)$
consists of two components. One is $D_{\bf a}(E)$ whose general points parameterize $n$ distinct points $p_1, \ldots, p_n$ in $E$ such that $\sum_{i=1}^n a_i p_i = 0$ and the other is $\delta_{0; \{1,\ldots,n\}}(E)$ whose general points parameterize $E$ attached to a rational tail that contains all the marked points. 

Now let us prove Theorem~\ref{thm:main} for the case $n\geq 4$. 

\begin{proof}
Note that for $\eta\neq 0$, $D_{\bf a}(E, \eta)$ is disjoint from $D_{\bf a}$. This is clear in the interior of $\BM_{1,n}$. At the boundary, if $k$ marked points coincide, say $p_1 = \cdots =  p_k = q$ in $E$, then 
$$\Big(\sum_{i=1}^k a_i\Big) q + \sum_{j=k+1}^n a_j p_j$$ 
has to be $\eta$ for $D_{\bf a}(E, \eta)$ and $0$ for $D_{\bf a}$, which cannot hold simultaneously for $\eta \neq 0$. 

Since $n\geq 4$, take $n-3$ very ample divisors on $\BM_{1,n}$ and consider their intersection restricted to $D_{\bf a}(E, \eta)$, which gives rise to an irreducible curve $C_{{\bf a}}(E, \eta)$ moving in $D_{\bf a}(E, \eta)$. Restricting to $D_{\bf a}(E, 0)$, we see that $C_{{\bf a}}(E, \eta)$ specializes to $C_{{\bf a}}(E, 0)$ which consists of two components $C_{{\bf a}}(E)$ and $C_{0;\{1,\ldots, n\}}(E)$, 
contained in $D_{\bf a}(E)$ and $\delta_{0; \{1,\ldots,n\}}(E)$, respectively. Moreover, $C_{{\bf a}}(E, 0)$ is connected, hence 
$C_{{\bf a}}(E)$ and $C_{0;\{1,\ldots, n\}}(E)$ intersect each other. Therefore, we conclude that
$$ (C_{{\bf a}}(E) + C_{0;\{1,\ldots, n\}}(E)) \cdot D_{\bf a} = C_{{\bf a}}(E, \eta)\cdot D_{\bf a} = 0, $$
$$ C_{0;\{1,\ldots, n\}}(E) \cdot D_{\bf a} > 0, $$
$$ C_{{\bf a}}(E) \cdot D_{\bf a} < 0. $$

The curve $C_{{\bf a}}(E)$ is not only moving in $D_{\bf a}(E)$ but also varies with the complex structure of $E$, hence it is moving in $D_{\bf a}$. Since it has negative intersection with $D_{\bf a}$ and $D_{\bf a}$ is irreducible, by Lemma~\ref{lem:negative} we thus conclude that $D_{\bf a}$ is extremal and rigid. 
\end{proof}

\begin{corollary}
\label{cor:moridream}
For $n\geq 3$ the moduli space $\BM_{1,n}$ is not a Mori dream space. 
\end{corollary}

\begin{proof}
By \cite[1.11 (2)]{HuKeel}, if $\BM_{1,n}$ is a Mori dream space, its effective cone would be the affine hull spanned by finitely many effective divisors, which contradicts the fact that $\BEff(\BM_{1,n})$ has infinitely many extremal rays. 
\end{proof}

%%%%%%%%%%%%%%%%%%%%%%%%%%%%
%%%%%%%%%%%%%%%%%%%%%%%%%%%%
\section{Effective divisors on $\TMM_{1,n}$} 
\label{sec:unordered}
%%%%%%%%%%%%%%%%%%%%%%%%%%%%
%%%%%%%%%%%%%%%%%%%%%%%%%%%%

In this section, we study the moduli space $\TMM_{1,n}$ of stable genus one curves with $n$ unordered marked points. 
The symmetric group $\mathfrak{S}_n$ acts by permuting the labeling of the points on $\BM_{1,n}$. We denote  the quotient $\BM_{1,n} / \mathfrak{S}_n $ by $\TMM_{1,n}$. 
The rational Picard group of $\TMM_{1,n}$ is generated by $\widetilde{\delta}_{\irr}$ and $\widetilde{\delta}_{0; k}$ for $2\leq k \leq n$, where $\widetilde{\delta}_{\irr}$ is the image of $\delta_{\irr}$ and $\widetilde\delta_{0; k}$ is the image of the union of $\delta_{0;S}$ for all $|S| = k$. 

In the case of genus zero, Keel and M$^{\rm c}$Kernan \cite{KeelMcKernanContractible} showed that the effective cone of $\TMM_{0,n}$ is spanned by the boundary divisors. Here we establish a similar result for $\TMM_{1,n}$. 

\begin{theorem}
\label{thm:unordered}
The effective cone of $\TMM_{1,n}$ is the cone spanned by the boundary divisors $\widetilde{\delta}_{\irr}$ and $\widetilde{\delta}_{0; k}$ for $2\leq k \leq n$.
\end{theorem}

\begin{proof}
It suffices to show that any irreducible effective divisor is a nonnegative linear combination of boundary divisors. Suppose  $D$ is an effective divisor different from any boundary divisor and has class
$$ D = a \widetilde{\delta}_{\irr} + \sum_{k=2}^n b_k \widetilde{\delta}_{0; k}. $$ 
If $C$ is a curve class whose irreducible representatives form a Zariski dense subset of a boundary divisor $\widetilde{\delta}_{0;k}$, then $C \cdot D \geq 0$. Otherwise, the curves in the class $C$ and, consequently, the divisor $\widetilde{\delta}_{0;k}$ would be contained in $D$, contradicting the irreducibility of $D$. We first show that $b_k \geq 0$ by induction on $k$. Here the argument is exactly as in Keel and M$^{\rm c}$Kernan and does not depend on the genus $g$.

Let $C$ be the curve class in $\TMM_{1,n}$ induced by fixing a genus one curve $E$ with $n-1$ fixed marked points and letting an $n$-th point vary along $E$. Since the general $n$-pointed genus one curve occurs on a representative of $C$, $C$ is a moving curve class. We conclude that $C \cdot D \geq 0$ for any effective divisor. On the other hand, since $C \cdot \widetilde{\delta}_{0;2} = n-1$ and $C \cdot \widetilde{\delta}_{\irr} = C \cdot \widetilde{\delta}_{0;k} = 0$, for $2< k \leq n$, we conclude that $b_2 \geq 0$. 

By induction assume that $b_k \geq 0$ for $k \leq j$. We would like to show that $b_{j+1} \geq 0$. Let $E$ be a genus one curve with $n-j$ fixed points. Let $R$ be a rational curve with $j+1$ fixed points $p_1, \dots, p_{j+1}$. Let $C_j$ be the curve class in $\TMM_{1,n}$ induced by attaching $R$ at $p_{j+1}$ to a varying point on $E$. Since the general point on $\widetilde{\delta}_{0;j}$ is contained on a representative of the class $C_j$, we conclude that $C_j$ is a moving curve in $\widetilde{\delta}_{0;j}$. Hence, $C_j \cdot D \geq 0$. 
On the other hand, $C_j$ has the following intersection numbers with the boundary divisors:
$$ C_j \cdot \widetilde{\delta}_{\irr} = 0, $$
$$ C_j\cdot \widetilde{\delta}_{0; i} = 0 \ \mbox{for}\ i \neq j, j+1, $$
$$ C_j \cdot \widetilde{\delta}_{0; j+1} = n-j, $$
$$  C_j \cdot \widetilde{\delta}_{0; j} = -(n-j).$$
Hence, we conclude that 
$ b_{j+1} \geq  b_j \geq 0$ by induction. Note that by replacing $E$ by a curve $B$ of genus $g$, we would get the inequalities $b_2 \geq 0$ and $(n-j) b_{j+1} \geq (2g -2 + (n-j))b_j$ for the coefficients of $\widetilde{\delta}_{0;k}$ on any non-boundary, irreducible effective divisor on $\TMM_{g,n}$.

There remains to show that the coefficient $a$ is non-negative. Fix a general pencil of plane cubics and a rational curve $R$ with $n+1$ fixed marked points $p_1, \dots, p_{n+1}$. Let $C_n$ be the curve class in $\TMM_{1,n}$ induced by attaching $R$ at $p_{n+1}$ to a base-point of the pencil of cubics. The class $C_n$ is a moving curve class in the divisor $\widetilde{\delta}_{0;n}$. Consequently, $C_n \cdot D \geq 0$. Since $C_n \cdot \widetilde{\delta}_{\irr} = 12$, $C_n \cdot \widetilde{\delta}_{0;k} = 0$ for $k < n$ and $C_n \cdot \widetilde{\delta}_{0;n} = -1$, we conclude that $12 a \geq b_n \geq 0$. This concludes the proof that the effective cone of $\TMM_{1,n}$ is generated by boundary divisors.  
\end{proof}

%%%%%%%%%%%%%%%%%%%%%%
%%%%%%%%%%%%%%%%%%%%%%
\begin{appendix}
%%%%%%%%%%%%%%%%%%%%%%
%%%%%%%%%%%%%%%%%%%%%%

\section{Singularities of $\BM_{1,n}$}

Let $\bM_{1,n}$ be the underlying course moduli scheme of $\BM_{1,n}$. Denote by 
$\bM_{1,n}^{\reg}$ its smooth locus. Below we will show that a canonical form defined on $\bM_{1,n}^{\reg}$
extends holomorphically to any resolution of $\bM_{1, n}$. 

Since $\bM_{1,n}$ is rational when $n \leq 10$ \cite{BelorousskiThesis}, in this case there are no non-zero holomorphic forms on any resolution. We may, therefore, assume that $n \geq 11$ as needed. The standard reference on the singularities of $\bM_{g,n}$ dates back to \cite{HarrisMumfordKodaira} and some recent generalizations include \cite{LoganKodaira,  LudwigSpin, FarkasVerraJacobian, ChiodoFarkasLevel, BiniFontanariVivianiPicard}.  

Let $(C; \bx) = (C; x_1, \ldots, x_n)$ be a stable curve with $n$ ordered marked points.  Let $\phi$ be a non-trivial automorphism of $C$ such that $\phi(x_i) = x_i$ for all $i$, and suppose that the order of $\phi$ is $k$. 
If the eigenvalues of the induced action of $\phi$ on $H^0(C, \omega_C \otimes \Omega_C^1(x_1 + \cdots x_n))^{\vee}$ are $e^{2\pi i k_j/k}$ with $0 \leq k_j < k$, then the age of $\phi$ is defined as  
$$\age(\phi) = \sum_{j} \frac{k_j}{k}.$$ 

If $\phi$ acts trivially on a codimension-one 
subspace of the deformation space of $(C; \bx)$, we say that $\phi$ is a quasi-reflection. For a quasi-reflection, all but one of the eigenvalues of $\phi$ are equal to one and $\age(\phi) = 1/k$.  By the Reid-Tai Criterion, see e.g. \cite[p. 27]{HarrisMumfordKodaira}, if $\age(\phi)\geq 1$ for any $\phi \in \Aut(C;\bx)$, then a canonical form defined 
on the smooth locus of the moduli space extends holomorphically to any resolution. Moreover, suppose that $\Aut(C;\bx)$ does not contain any quasi-reflections, then the resulting singularity is canonical if and only if $\age(\phi)\geq 1$ for any $\phi \in \Aut(C,\bx)$, see e.g. \cite[Theorem 3.4]{LudwigSpin}.  The quasi-reflections  form a normal subgroup of $\Aut(C,\bx)$. One can consider the action modulo this subgroup and use the Reid-Tai Criterion, see
\cite[Proposition 3.5]{LudwigSpin}. In particular, no singularities arise if and only if $\Aut(C,\bx)$ is generated by 
quasi-reflections. 

The automorphism $\phi$ induces an action on 
$H^0(C, \omega_C\otimes \Omega_C^1(x_1+\cdots + x_n))^{\vee}$. We have an exact sequence: 
$$ 0 \to \bigoplus_{p\in C_{\sing}} \tor_p \to  H^0(C, \omega_C\otimes \Omega_C^1(x_1+\cdots + x_n)) \to 
\bigoplus_{\alpha} H^0(C_{\alpha}, \omega_{C_\alpha}^{\otimes 2}(\sum_{\beta} p_{\alpha\beta})) \to 0, $$
where $C_{\alpha}$'s are the components of the normalization of $C$ and $p_{\alpha\beta}$'s are the inverse images of nodes in $C_\alpha$.   

First, we show that for an irreducible elliptic curve $E$ with $n$ distinct marked points, 
we have $\age(\phi) \geq 1$. The automorphism group of $E$ has order 2 if $j(E) \not= 0, 1728$, has order $4$ if $j(E) = 1728$, and has order $6$ if $j(E) = 0$. Since $\phi$ fixes all $x_1, \ldots, x_n$, if $n\geq 3$, then $\phi$ has order $k = 2$ or  $3$. If $k = 2$, then $n= 3$ or  $4$, and hence by \cite[p. 37, Case c2)]{HarrisMumfordKodaira} 
we have $\age(\phi) = \frac{n-1}{2} \geq 1$. If $k = 3$, then $n = 3$, and hence \cite[p. 38, Case c3)]{HarrisMumfordKodaira} implies that $\age(\phi) \geq 1$.  
 
Next, consider a stable nodal genus one curve $(C; \bx)$ with $n$ ordered marked points. Let $C_0$ be its core curve of genus one. Then $C_0$ is either irreducible elliptic, or consists of a circle of $\bbP^1$s. It is easy to see that $\phi$ acts trivially on every component of $C\backslash C_0$. Let $C_0$ be a circle of $l$ copies of $\bbP^1$, i.e. $B_1,\ldots, B_l$ are glued successively at the nodes $p_1, \ldots, p_l$, where $B_i\cong \bbP^1$, $B_i\cap B_{i+1} = p_{i+1}$ and $p_{l+1} = p_1$. By the stability of $(C; \bx)$, each $B_i$ contains at least one more node or marked point, which has to be fixed by $\phi$. 
Therefore, $\phi$ acts non-trivially on $B_i$ only if it acts as an involution, switching $p_i$ and $p_{i+1}$ and fixing the other nodes and marked points on $B_i$. This implies that $l = 2$ and $k = 2$. By \cite[p. 34]{HarrisMumfordKodaira},  either $\age(\phi) \geq 1$ or $\Aut(C,\bx)$ is generated by this elliptic involution, which is a quasi-inflection and does not induce a  singularity. Thus, we are left with the case when $C_0$ is an irreducible elliptic curve $E$ and $\phi$ is induced by a non-trivial automorphism of $E$ fixing all marked points and acting trivially on the other components of $C$. 

If $E$ contains at least one marked point $x$, \cite[proof of Theorem 1.1 (ii)]{FarkasVerraJacobian} says that $\age(\phi) \geq 1$. We can also see this directly using \cite[p. 37-39, Case c)]{HarrisMumfordKodaira} as follows. If the order $n$ of $\phi$ 
is $2$, then the action restricted to $H^0(K_E^{\otimes 2}(x))$ contributes $1/2$ to $\age(\phi)$. At a node $p$ of $E$, suppose that the two branches have coordinates $y$ and $z$. Then $\tor_p$ is generated by 
$y dz^{\otimes 2}/z = z dy^{\otimes 2}/y$, see \cite[p. 33]{HarrisMumfordKodaira}. The action of $\phi$ locally is given by $y\to -y$ and $z\to z$, hence $\tor_p$ also contributes $1/2$. Consequently we get $\age(\phi) \geq 1$. 
If $k = 3$, at $p$ the action is locally given by $y\to \zeta y$ and $z\to z$, where $\zeta$ is a cube root of unity, hence $\tor_p$ contributes $1/3$. At $x$, take a translation invariant differential $dz$. Then locally $dz^{\otimes 2}$ is an eigenvector of $H^0(K_E^{\otimes 2}(x+p))$. The action $\phi$ is locally given by $x\to \zeta x$, hence it contributes $2/3$. We still get $\age(\phi)\geq 1/3 + 2/3 = 1$. If $k=4$, similarly $\tor_p$ contributes $1/4$. Locally take $dz^{\otimes 2}$ and $dz^{\otimes 2}/z$ as eigenvectors of 
 $H^0(K_E^{\otimes 2}(x+p))$. We get an additional contribution equal to $2/4 + 1/4$. In total we still have $\age(\phi)\geq 1$. Finally, since $\phi$ cannot fix both $x$ and $p$, the case $k=6$ does not occur. Similarly, if $E$ contains more than one node, $\phi$ fixes all the nodes, and hence the same analysis implies that $\age(\phi) \geq 1$. 

Based on the above analysis, we conclude that the locus of non-canonical singularities of $\bM_{1,n}$ is contained in the locus of curves $(C,\bx)$ where the core curve of $C$ is an unmarked irreducible elliptic tail $E$ attached to the rest of $C$ at a node $p$. Moreover, $G = \Aut(C,\bx) = \Aut(E,p)$ fixes all marked points and acts trivially on the other components of $C$. Harris and Mumford \cite[p. 40-42]{HarrisMumfordKodaira} proved that any canonical form defined in $\bM_{g,n}^{\reg}$ extends holomorphically to any resolution over the locus of curves of this type. Strictly speaking, Harris and Mumford discussed the case $\bM_{g}$. They constructed a suitable neighborhood of a point in $\bM_g$ parameterizing an elliptic curve attached to a curve $C_1$ of genus $g-1$ without any automorphisms. In their construction, the only property of $C_1$ they need is that  $C_1$ does not have any non-trivial automorphisms. Hence, their construction is applicable to the case when $C_1$ is replaced by an arithmetic genus zero curve with $n$ marked points for $n\geq 2$. Therefore, there is a neighborhood of $(C,\bx)$ in $\bM_{1,n}$ such that any canonical form defined in the smooth locus of this neighborhood extends holomorphically to a desingularization of the neighborhood. This thus completes the proof.  

\end{appendix}

%%%%%%%%%%%%%%%%%%%
%%%%%%%%%%%%%%%%%%%

\end{document}